\documentclass[12pt]{article}
\usepackage{amssymb,amscd,amsmath,amsthm} 
\usepackage{latexsym}
\usepackage{graphicx}
\usepackage{amssymb}

\DeclareFontFamily{OT1}{pzc}{}
\DeclareFontShape{OT1}{pzc}{m}{it}{<-> s * [1.10] pzcmi7t}{}
\DeclareMathAlphabet{\mathpzc}{OT1}{pzc}{m}{it}

\theoremstyle{plain}

\newtheorem{thm}{Theorem}[section]
\newtheorem{lem}[thm]{Lemma}
\newtheorem{case}{Case}
\newtheorem{prop}[thm]{Proposition}

\begin{document} 

\pagenumbering{arabic} \setcounter{page}{1}

\title{On the sum of dilations of a set}

\author{Antal Balog\thanks{The first author's research was supported by the Hungarian National Science Foundation Grants K81658 and K104183.} \\ 
Alfr\'ed R\'enyi Institute of Mathematics \\
Budapest, P.O.Box 127, 1364--Hungary \\
balog@renyi.mta.hu \\
\and
George Shakan\\
University of Wyoming Department of Mathematics \\
Laramie, Wyoming 82072, USA \\
gshakan@uwyo.edu}

\maketitle


\begin{abstract}
We show that for any relatively prime integers $1\leq p<q$ and for any finite $A \subset \mathbb{Z}$ one has $$|p \cdot A + q \cdot A | \geq (p + q) |A| - (pq)^{(p+q-3)(p+q) + 1}.$$
\end{abstract}

\section{Introduction}

Let $A$ and $B$ be finite sets of real numbers. The sumset and the product set of $A$ and $B$ are defined by $$A + B = \{a + b : a \in A, \ b \in B \},$$
$$A \cdot B = \{ab : a \in A, \ b \in B\}.$$ For $d > 0$ the dilation of $A$ by $d$ is defined by 
$$d \cdot A = \{d\} \cdot A =  \{da : a \in A\},$$ while for any real number $x$, the translation of $A$ by $x$ is defined by 
$$x + A = \{x\} + A = \{x + a : a \in A\}.$$ 

The Erd\H os--Szemer\'edi sum--product conjecture \cite{Er} claims that for any finite subset of the positive integers, either the sumset $A+A$ or the product set $A\cdot A$ must be big, more precisely
$$\max\{|A+A|,|A\cdot A|\}\gg |A|^{2-\epsilon}$$ 
for any $\epsilon>0$. The best result in this direction, due to Solymosi \cite{So}, is the above bound with the weaker exponent $4/3-\epsilon$. Another realization of this phenomenon is if an expression of sets uses both addition and multiplication, then it produces a big set. For example, 
$$|A\cdot A + A|\gg |A|^{3/2},\quad |(A+A) \cdot A|\gg |A|^{3/2}$$ 
both come from the method of Elekes \cite{El} using the Szemer\'edi--Trotter incidence geometry, see the book of Tao and Vu \cite{Ta}. Improving the exponent in these bounds is a challenging problem, and $2-\epsilon$ is certainly expected. Such a problem seems easier for more variables, for example
$$|A\cdot A + A\cdot A + A\cdot A + A\cdot A |\geq\frac12  |A|^2$$
is proved by the first author \cite{Ba}. Changing the role of addition and multiplication in most of these expressions does not change the results or our expectation dramatically. 

However, the two variable expressions $A\cdot(1+A)$ and $p \cdot A+q\cdot A$ are exceptions, because translation seriously alters multiplicative behavior, while dilation seems rather harmless. It is a beautiful consequence of the incidence geometry that $A\cdot(1+A)$ is big; for example, Jones and Roche--Newton \cite{Jo} proved 
$$|A\cdot (1+A)|\gg|A|^{24/19-\epsilon}.$$
On the other hand, let $q > p \geq 1$ be relatively prime integers and for $X=\{1,2,\dots,|X|\}$, obviously $p\cdot X + q\cdot X\subset \{p+q,\dots, (p+q)|X|\}$, that is 
\begin{equation}|p\cdot X+q\cdot X|\leq (p+q)|X|-(p+q-1).\label{triviupper}\end{equation} 

Bukh \cite{Bu} proved that for coprime integers $\lambda_1 , \ldots , \lambda_k$ one has $$|\lambda_1 \cdot A + \ldots + \lambda_k \cdot A| \geq (\lambda_1 + \ldots + \lambda_k) |A| - o(|A|).$$ Our main result says $|p \cdot A + q \cdot A| \geq (p+q)|A| - C_{p,q}$. Cilleruelo, Hamidoune and Serra showed this for $p=1$ and $q$ prime and these results were extended by Du, Cao, and Sun \cite{Du} when $q$ is a prime power or the product of two primes. Hamidoune and J. Ru\'{e} \cite{Ha} solved the case when $p=2$ and $q$ prime and these results were extended by Ljujic \cite{Lj} to $p=2$ and $q$ the power of an odd prime or the product of two odd primes. While it is certainly feasible that $|\lambda_1 \cdot A + \ldots + \lambda_k \cdot A| \geq (\lambda_1 + \ldots + \lambda_k) |A| - C$ for some $C$ that only depends on $\lambda_1, \ldots , \lambda_k$, our method seems to only handle the case where $k \leq 2$. 

We will confine ourselves to subsets of the integers. Transforming our result to the case when $p,\, q$ and the set $A$ are from rationals is an obvious task. As Sean Eberhard pointed out, keeping $p$ and $q$ integers and extending our result to where $A$ is a subset of the real numbers follows from, say, Lemma 5.25 in \cite{Ta}. Our main result is as follows. 

\begin{thm} \label{main}  For any relatively prime integers $1\leq p<q$ and for any finite $A \subset \mathbb{Z}$ one has 
$$|p \cdot A + q \cdot A | \geq (p + q) |A| - (pq)^{(p+q-3)(p+q) + 1}.$$
\end{thm}

Note that the multiplicative constant $p+q$ is best possible as the above example (\ref{triviupper}) shows. On the other hand  we do not attempt to get the best additive constant, and improvements in that respect are very probable. For example, our method gives $q2^{(q-2)(q+1)}$ to the case $p=1$, or simply $q!$ to the case $p=1$, $q$ is a prime. We cannot even decide the true size of the best additive constant. The following example shows that, even in the special case $p=1$, the additive constant is not polynomial in $q$. It also suggests that possibly a better constant can be proved for all sufficiently large sets $A$, avoiding such pathological cases. 


For $t$ a positive integer and $0 \leq a < \frac{q}{2}$ also an integer, let $$A : = \{a_0 + a_1 q + a_2 q^2 + \ldots + a_t q^t : 0 \leq a_i < a , \ a_i \in \mathbb{Z} \}.$$ If we set $a = \lfloor \sqrt{q} \rfloor$ and $t = \lfloor \log_2{\sqrt{q}} \rfloor$, it is easy to see that $|A|=a^{t+1}$ and $|A+q\cdot A|=a^2(2a-1)^t\leq a2^t|A|\leq q|A|$. Thus $|A+q\cdot A|\leq(q+1)|A|-(\sqrt{q}-1)^{\log_2{(\sqrt{q}-1})}$. The authors are thankful to Imre Ruzsa for drawing their attention to the present problem, as well as for pointing out this example. 

\bigskip\bigskip

\section{Preliminaries}\label{q is 3}

For nonempty finite subsets of the real numbers $A = \{a_1 < \ldots < a_r\}$ and $B = \{b_1 < \ldots < b_s\}$, we have that 
\begin{equation}
|A + B| \geq |A| + |B| - 1, 
\label{trivial}
\end{equation} by observing that $$a_1 + b_1 < a_2 + b_1 < \ldots < a_r + b_1 <a_r + b_2 < \ldots < a_r + b_s.$$ We extend this argument in the following lemma.

\begin{lem}\label{base} Let $A$ be a nonempty subset of the real numbers and $q > p \geq 1$. Then $$|p \cdot A + q \cdot A| \geq 3|A| - 2.$$ \end{lem}

\begin{proof} Let $A = \{a_1, \ldots , a_n\}$ where $a_1 < a_2 < \ldots < a_n$. Then $$p a_1 + q a_1 <p a_2 + q a_1 < p a_1 + q a_2 < p a_2 + q a_2 < p a_3 + q a_2 < \ldots < p a_{n-1} + q a_n < p a_n + q a_n.$$ For each $1 \leq i \leq n-1$ there are three elements in the previous list, $p a_i + q a_i < p a_{i+1} + q a_i < p a_i + q a_{i+1}$, and one more element $p a_n + q a_n$. So $|p\cdot A + q \cdot A| \geq 3(n - 1) + 1 = 3 |A| -2 $. 

\end{proof}

It is worth noting that this gives $|A + 2 \cdot A| \geq 3 |A| - 2$, which settles Theorem \ref{main} in the case $p=1,\,q=2$. This is best possible by (\ref{triviupper}).

The main purpose of this section is to give a short proof of $|A + 3 \cdot A| \geq 4|A| - 4$ in order to introduce some of the main ideas of the proof of the general theorem. This was proved by Cilleruelo, Silva, and Vinuesa \cite{Si} using different methods and they were able to classify the sets where equality holds. The method we use requires that $A$ is a subset of the integers, and can be extended to a general lower bound for all cases with $q\geq 3$, though we leave this extension to the interested reader. The bound $4|A|-4$ is also best possible, as is follows from the next construction. Let $0\leq d < q\leq n$ be  integers and let $X := \{i + xq : 0\leq i\leq d,\, 0 \leq x <n\}$. Note that $|X|=(d+1)n$. It is easy to check that $X + q \cdot X$ is precisely the set of integers in the interval $[0, \ldots , (q+1)(d + (n-1)q)]$ that are equivalent to one of $\{0 , \ldots , d\}$ modulo $q$. It follows that $|X + q \cdot X| = (q+1)|X| - (d+1) (q - d)$, and the optimal choice $d=\lfloor (q-1)/2\rfloor$ gives $|X+q\cdot X| =(q+1)|X| - \lfloor \frac{q+1}{2} \rfloor \lceil \frac{q+1}{2} \rceil$. 

We would like to remark here one can use \ref{main} and the methods of \cite{Ci} to show for $|A| \geq (q-1)^2 q^{q^2 - q+1}$, one has $|A + q \cdot A| \geq (q+1)|A| -  \lfloor \frac{q+1}{2} \rfloor \lceil \frac{q+1}{2} \rceil$. Furthermore, their methods can be used to show $X$  is the unique set, up to an affine transformation, where equality holds.

\begin{thm}\label{q=3} Let $A$ be a finite subset of the integers. Then $|A + 3 \cdot A| \geq 4 |A| - 4$.
\end{thm}

\begin{proof}

If $|A| =1$ the result is trivial. If $|A|=2$ the result follows from Lemma \ref{base}. We use induction on the size of $|A|$.  Assume  $|A|>2$, and  for any proper subsets $X$ of $A$, we have $|X + 3 \cdot X| \geq 4 |X| - 4$. 

Translation and dilation do not change $|A + 3 \cdot A|$, we may translate $A$ so that 0 is the smallest element and then dilate $A$ until it intersects at least 2 residue classes modulo 3. Let $$A = \displaystyle \bigcup_{j = 1}^r A_j, \ \ A_j = a_j + 3 \cdot B_j  , \ \ B_j \neq \emptyset, \ \ 0 \leq a_j < 3.$$

This union is disjoint and moreover the sets $A_j+3\cdot A$ are disjoint. Thus from the obvious fact \eqref{trivial} it follows that \begin{equation} \label{FD3} |A+ 3 \cdot A| = \displaystyle \sum_{j = 1}^r |A_j + 3 \cdot A| \geq \displaystyle \sum_{j = 1}^r  \big(|A_j| + |A| - 1\big) = (r+1)|A| - r.\end{equation}

We say that $A$ is fully distributed modulo 3 (FD mod 3) if $A$ intersects all 3 residue classes modulo 3. By \eqref{FD3}, if $A$ is FD mod 3 then $|A + q \cdot A| \geq 4|A| - 3 > 4|A| - 4$ and the theorem follows. 

Thus we may assume that $A$ intersects exactly two residue classes modulo 3, so $A = A_1 \cup A_2$. Then $$A+ 3 \cdot A = (A_1 + 3 \cdot A) \cup (A_2 + 3 \cdot A)$$ where the union is disjoint and $$ |A_1 + 3 \cdot A | = |A_1+ 3 \cdot A_1| + |(  A_1 + 3 \cdot A_2) \setminus (A_1 +  3 \cdot A_1)|,$$ 
$$|A_2 + 3 \cdot A | = |A_2+ 3 \cdot A_2| + |(  A_2  + 3 \cdot A_1) \setminus  (A_2 +  3 \cdot A_2)|.$$ Suppose now that $$|(A_1 + 3 \cdot A_2) \setminus (A_1 + 3 \cdot A_1)| < |A_2|.$$ After a translation by $-a_1$, dilation by $\frac{1}{q}$, and another translation by $-a_1$, we obtain $$|(a_2 - a_1 + B_1 + 3 \cdot B_2) \setminus (B_1 + 3 \cdot B_1)| < |B_2|.$$ Therefore for any $x \in B_1$, there is a $y \in B_2$ such that $a_2-a_1+x+3y\in B_1+3\cdot B_1$ and so there is an $x'\in B_1$ such that $a_2 - a_1 + x \equiv x' \pmod 3$. We may repeat this for $x'$ in place of $x$ so there is an element $x'' \in B_1$ such that $a_2 - a_1 + x' \equiv x'' \pmod 3$. Since $(a_2 - a_1,3)=1$, we have that $x , x' , x''$ are incongruent mod 3. Thus $B_1$ is FD mod 3 and it follows from \eqref{FD3} that $$|A_1 + 3 \cdot A_1| = |B_1 + 3 \cdot B_1| \geq 4 |B_1| - 3 = 4 |A_1| - 3.$$ A symmetric argument shows if $|(A_2 + 3 \cdot A_1) \setminus (A_2 + 3 \cdot A_2)| < |A_1|$ then we have that $B_2$ is FD mod 3 and that $|A_2 + 3 \cdot A_2| \geq 4|A_2| - 3$.

\begin{case}
Suppose that $B_1$ and $B_2$ are both FD mod 3. 
\end{case}
We can always find two more elements in the (disjoint) union of $(A_1 + 3 \cdot A_2) \setminus (A_1 + 3 \cdot A_1)$ and $(A_2 + 3 \cdot A_1) \setminus ( A_2 + 3 \cdot A_2)$. Consider the maximum element of $A$ and let it be in $A_2$ and call it $M$. Let $u \in A_1$ be maximally chosen. Then it is clear that $u + 3M \in (A_1 + 3 \cdot A_2) \setminus (A_1 + 3 A_1)$. A symmetric argument works if $M \in A_1$. Call this new element $z_1$. Similarly let $m \in A$ be minimal and assume $m \in A_1$. Then let $v \in A_2$ be minimally chosen. It is clear that $v + 3 m \in (A_2 + 3 \cdot A_1)  \setminus (A_2 + 3 \cdot A_2)$. A symmetric argument works if $m \in A_2$. Call this new element $z_2$. Since $v+3m\leq M+3m<m+3M\leq u+3M$, then $z_1 \neq z_2$ and we have 
$$\{z_1,z_2 \} \subset (A_1 + 3 \cdot A_2) \setminus (A_1 + 3 \cdot A_1) \cup (A_2 + 3 \cdot A_1) \setminus ( A_2 + 3 \cdot A_2).$$
Then by \eqref{FD3} $$|A+ 3 \cdot A| \geq |A_1 + 3 \cdot A_1| + |A_2 + 3 \cdot A_2| + |\{z_1 , z_2\}| \geq 4 |A_1| - 3 + 4 |A_2| - 3 + 2 = 4|A| - 4.$$ 

\begin{case}
Suppose that $B_1$ and $B_2$ are both not FD mod 3. \end{case} Then $|(A_1 + 3 \cdot A_2) \setminus (A_1 + 3 \cdot A_1)| \geq |A_2|$ and $|(A_2 + 3 \cdot A_1) \setminus (A_2 + 3 \cdot A_2)| \geq |A_1|$ and we have from Lemma \ref{base} that $$|A + 3 \cdot A| = |A_1 + 3 \cdot A| + |A_2 + 3 \cdot A| \geq 3|A_1| - 2 + |A_2| + 3 |A_2| - 2 + |A_1| = 4 |A| - 4.$$

\begin{case} Suppose one of $B_1$ and $B_2$ is FD mod 3 and the other is not. \end{case} We may assume, without loss of generality, that $B_1$ is FD mod 3 while $B_2$ is not FD mod 3. This is the only case when we use the induction hypothesis. Since $B_1$ is FD mod 3, $|A_1| \geq 3$ and since $B_2$ is not FD mod 3, $|(A_2 + 3 \cdot A_1) \setminus (A_2 + 3 \cdot A_2)| \geq |A_1| \geq 3$. By induction and \eqref{FD3}, we have $$|A + 3 \cdot A| = |A_1 + 3 \cdot A| + |A_2 + 3 \cdot A| \geq 4 |A_1| - 3 + 4 |A_2| - 4 + |A_1| \geq 4 |A| - 4.$$

In all cases we obtain $|A + 3 \cdot A| \geq 4|A| - 4$, thus completing the induction.

\end{proof}

\section{Proof of Theorem \ref{main}}

Translation and dilation do not change $|p \cdot A + q \cdot A|$. 

Suppose the residue classes modulo $p$ that intersect $A$ are $p_1 , \ldots, p_r$ and suppose the residue classes modulo $q$ that intersect $A$ are $q_1 , \ldots , q_s$. For $1 \leq i \leq r$ and $1 \leq j \leq s$, consider the greatest common divisors $$d_{p,i} = (p_1 - p_i , p_2 - p_i , \ldots , p_r - p_i , p),$$ $$d_{q,j} = (q_1 - q_j , q_2 - q_j , \ldots , q_s - q_j , q).$$ If there is a $d_{q,j} > 1$, then change $A$ to $(A - q_j) / d_{q,j}$. Similarly if there is a $d_{p,i} > 1$, then change $A$ to $(A - p_i)/d_{p,i}$. Note that such a change redefines the residue classes $p_i,\,q_j$ as well as $r$ and $s$.

Repeat this process as many times as possible. If $|A| \geq 2$ then the process must stop in finitely many steps since each reduction decreases the distance between the minimal and maximal elements of $A$. 

We say that $A$ is reduced if both of the following hold for any $1 \leq i \leq r$ and $1 \leq j \leq s$:

\begin{equation} (p_1 - p_i , p_2 - p_i , \ldots , p_r - p_i , p) = 1 \label{p}, \end{equation}

\begin{equation} (q_1 - q_j , q_2 - q_j , \ldots , q_s - q_j , q) = 1 \label{q}. \end{equation}

Observe that any reduced set must have at least two residue classes modulo $q$ and modulo $p$ when $p > 1$. 


We will assume $A$ satisfies \eqref{p} and \eqref{q}. We split $A$ into residue classes modulo $p$ and modulo $q$ as follows:

$$A = \displaystyle \bigcup_{i = 1}^r P_i , \ \  P_i = p_i + p \cdot P_i^{\prime}, \ \ P_i \neq \emptyset ,\ \  0 \leq p_i < p,$$

$$A = \displaystyle \bigcup_{j = 1}^s Q_j , \ \  Q_j = q_i + q \cdot Q_j^{\prime}, \ \ Q_j \neq \emptyset ,\ \  0 \leq q_j < q.$$

We would like to remark here that the proof is simpler when $p = 1$, since \eqref{p} is vacuous and we do not need to split $A$ into residue classes mod $p$. 

Note that \begin{align*} p \cdot A + q \cdot A &= \displaystyle \bigcup_{i =1}^r \big(p \cdot A + q \cdot P_i\big) = \bigcup_{j=1}^s \big(p \cdot Q_j + q \cdot A\big) \\ &= \displaystyle \bigcup_{i=1}^r \bigcup_{j=1}^s \big(p \cdot Q_j + q \cdot P_i\big)  \end{align*} where the unions are disjoint. Indeed, for any $1 \leq i \leq r$ and $1 \leq j \leq s$, any element of $p \cdot Q_j + q \cdot P_i$ is equivalent to $p q_j \pmod q$ and $q p_i \pmod p$ and since $(p,q) = 1$, the $p \cdot Q_j + q \cdot P_i$ are pairwise disjoint.  This is not true for, say $A + A + q \cdot A$, and the extension of our method to three or more terms does not seem 
straightforward.

 Utilizing \eqref{trivial}, we obtain

\begin{align}
\label{first}
|p \cdot A + q \cdot A|  &= \displaystyle \sum_{i= 1}^r \sum_{j=1}^s |p \cdot Q_j + q \cdot P_i|  \\ \notag
& \geq \displaystyle \sum_{i= 1}^r \sum_{j=1}^s \big(|Q_j| + |P_i| - 1\big)  = (r+s) |A| - rs.
\end{align}

We will say that $A$ is fully distributed modulo $p$ (FD mod $p$) if $A$ intersects every residue class modulo $p$ and $A$ is fully distributed modulo $q$ (FD mod $q$) if $A$ intersects every residue class modulo $q$. Thus $A$ is FD mod $p$ and/or FD mod $q$ if and only if $r=p$ and/or $s=q$. 

\begin{lem} \label{dist1} For any fixed $1 \leq j \leq s$, either $Q_j^{\prime}$ is FD mod $q$ or $$|p \cdot Q_j + q \cdot A| \geq |p \cdot Q_j + q \cdot Q_j| + \min_{1 \leq m \leq s} |Q_m|.$$ Similarly for any fixed $1 \leq i \leq r$, either $P_i^{\prime}$ is FD mod $p$ or $$|p \cdot A + q \cdot P_i| \geq |p \cdot P_i + q \cdot P_i| + \min_{1 \leq k \leq r} |P_k|.$$
\end{lem}

\begin{proof} We only prove the first statement and the second statement is a symmetric argument. Suppose $|p \cdot Q_j + q \cdot A| < |p \cdot Q_j + q \cdot Q_j| + \displaystyle \min_{1 \leq m \leq s} |Q_m|.$ Then for any $m \neq j$, we have $$|Q_m^{\prime}| = |Q_m| > |(p \cdot Q_j + q \cdot Q_m )\setminus (p \cdot Q_j + q \cdot Q_j) | = |(q_m - q_j + p \cdot Q_j^{\prime} + q \cdot  Q_m^{\prime}) \setminus (p \cdot Q_j^{\prime} + q \cdot Q_j^{\prime})|.$$

It follows that for every $x \in p \cdot Q_j^{\prime}$ there is a $y \in  Q_m^{\prime}$ such that $q_m-q_j+x+qy\in p\cdot Q'_j + q\cdot Q'_j$, and so there is an $x'\in p\cdot Q'_j$ such that $q_m - q_j + x \equiv x' \pmod q$. We may repeat this argument with $x'$ in place of $x$, and so on, and we may repeat for all $m$ so that eventually we get for any $x \in p \cdot Q_j^{\prime}$ and any $z = u_1 (q_1 - q_j) + \cdots +u_s (q_s - q_j)$, where $u_1 , \ldots , u_s$ are arbitrary integers, that there is an $x' \in p \cdot Q_j^{\prime}$ with $z + x \equiv x' \ (q)$. Since $A$ is reduced, the set of $z$ describes all residues mod $q$ by \eqref{q}, it follows that $p \cdot Q_j^{\prime}$ is FD mod $q$ and since $(p,q) =1$, $Q_j^{\prime}$ is FD mod $q$. 
\end{proof}

The previous lemma will be useful in finding new elements if any of the $P_i^{\prime}$ are not FD mod $p$ or if any of the $Q_j^{\prime}$ are not FD mod $q$. For convenience, set $A_{ij} := P_i \cap Q_j$, where some of these sets may be empty. Then 

$$A = \displaystyle \bigcup_{i = 1}^r \bigcup_{j=1}^s A_{ij}, \ \  A_{ij} = a_{ij} + pq \cdot A_{ij}^{\prime}, \ \  0 \leq a_{ij} <  pq , \ \ a_{ij} \equiv p_i (p) ,  \ \ a_{ij} \equiv q_j (q).$$

The fact that we may write $A_{ij} = a_{ij} + pq \cdot A_{ij}^{\prime}$ is precisely the Chinese remainder theorem. We have that \begin{equation} |A| = \displaystyle \sum_{i = 1}^r \sum_{j=1}^s |A_{ij}|, \label{sum} \end{equation} where some of the summands are possibly zero. 

\begin{lem}\label{dist2}

Fix $1 \leq i \leq r$ and $1 \leq j \leq s$. Suppose $P_i^{\prime}$ is FD mod $p$. Then either $A_{ij}^{\prime}$ is FD mod $p$ or $$|p \cdot Q_j + q \cdot P_i| \geq |p \cdot A_{ij} + q \cdot A_{ij}| + |A_{ij}|.$$ Similarly, suppose $Q_j^{\prime}$ is FD mod $q$. Then either $A_{ij}^{\prime}$ is FD mod $q$ or $$|p \cdot Q_j + q \cdot P_i| \geq |p \cdot A_{ij} + q \cdot A_{ij}| + |A_{ij}|.$$

\end{lem}

\begin{proof} We prove the first statement, and the second statement follows by a symmetric argument. The result is trivial for $A_{ij} = \emptyset$. For $A_{ij} \neq \emptyset$, we analyze 
$$|p \cdot A_{ij} + q \cdot P_i |= |p \cdot A_{ij} + q \cdot A_{ij} |+ |(p \cdot A_{ij} + q \cdot P_i) \setminus (p \cdot A_{ij} + q \cdot A_{ij} )|.$$ An easy calculation reveals that 
$$|(p \cdot A_{ij} + q \cdot P_i) \setminus (p \cdot A_{ij} + q \cdot A_{ij} )| = |(\frac{p_i - a_{ij}}{p} + p \cdot A_{ij}^{\prime} + P_i^{\prime} ) \setminus (p \cdot A_{ij}^{\prime} + q \cdot A_{ij}^{\prime})|.$$ 
Note that $\frac{p_i - a_{ij}}{p} \in \mathbb{Z}$.
Suppose this is smaller than $|A_{ij}| = |A_{ij}^{\prime}|$. Then for every $x \in P_i^{\prime}$ there is a $y \in A_{ij}^{\prime}$ such that $\frac{p_i - a_{ij}}{p} + p y + x \in p \cdot A_{ij}^{\prime} + q \cdot A_{ij}^{\prime}$. This means, for every  $x \in P_i^{\prime}$ there is  an $x'\in q\cdot A'_{ij}$ such that $\frac{p_i - a_{ij}}{p} + x \equiv x' \pmod p$ . Since $P_i^{\prime}$ is FD mod $p$, we have that $q \cdot A_{ij}^{\prime}$ is FD mod $p$. But $(p,q) = 1$, so $A_{ij}^{\prime}$ is FD mod $p$.

\end{proof}

We are now ready to prove Theorem \ref{main}. Our strategy is simple: we start from Lemma \ref{base} and gradually improve it in an iterative way.

\begin{prop} \label{proof} Let $q > p \geq 1$ are relatively prime integers. For every integer $3(p + q) \leq m \leq (p+q)^2$ and for all finite sets of integers $A$, we have $$|p \cdot A + q \cdot A| \geq \frac{m}{p+q} |A| - (pq)^{m + 1 - 3(p+q)}.$$
\end{prop}

\begin{proof} Observe that $m = (p+q)^2$ is precisely Theorem \ref{main}. We prove by induction on $m$. For $m = 3(p+q)$, we claim $|p \cdot A + q \cdot A| \geq 3 |A| - pq$, which is even true with $3 |A| - 2$ by Lemma \ref{base}.

Suppose now that Proposition \ref{proof} is true for a fixed $3(q+p) \leq m < (p+q)^2$. For simplicity, we write $$C_m = (pq)^{m + 1- 3(p+q)}.$$

First, assume there is an $1 \leq i \leq r$ such that $|P_i| \leq \frac{1}{p+q} |A|$. Then by the induction hypothesis, \eqref{trivial}, and $m \leq (p+q)^2$, we obtain

\begin{align*}
|p \cdot A + q \cdot A| &\geq |p \cdot A + q \cdot P_i| + |p \cdot (A \setminus P_i) + q \cdot (A \setminus P_i)| \geq \\
\geq & |P_i| + |A| - 1  + \frac{m}{p+q} (|A| - |P_i|) - C_m \geq \frac{m+1}{p+q} |A| - C_{m+1},
\end{align*} since $C_{m+1} \geq C_m + 1$. A symmetric argument completes the induction step if for some $1 \leq j \leq s$, $|Q_j| \leq \frac{1}{p+q} |A|$. Thus we may assume that every $P_i$ and $Q_j$ has more than $\frac{1}{p+q}|A|$ elements. If there is a $1 \leq i \leq r$ such that $P_i^{\prime}$ is not FD mod $p$ then by Lemma \ref{dist1} and the induction hypothesis we have

\begin{align*}
|p \cdot A + q \cdot A| &\geq |p \cdot A + q \cdot P_i| + |p \cdot (A \setminus P_i) + q \cdot (A \setminus P_i)| \\
&\geq  |p \cdot P_i + q \cdot P_i| + \min_{1 \leq k \leq r} |P_k| + \frac{m}{p+q} (|A| - |P_i|) - C_m \\
&\geq \frac{m}{p+q} |P_i|  - C_m + \frac{1}{p+q} |A| + \frac{m}{p+q} (|A| - |P_i|) - C_m \geq \frac{m+1}{p+q} |A| - C_{m+1},
\end{align*} since $C_{m+1} \geq 2 C_m$. A symmetric argument works if there is a $1 \leq j \leq s$ such that $Q_j^{\prime}$ is not FD mod $q$. Thus we may assume that every $P_i^{\prime}$ is FD mod $p$ and every $Q_j^{\prime}$ is FD mod $q$. 

Fix $1 \leq i \leq r$ and $1 \leq j \leq s$. Then using both parts of Lemma \ref{dist2} we obtain either $$|p \cdot Q_j + q \cdot P_i| \geq |p \cdot A_{ij} + q \cdot A_{ij}| + |A_{ij}| \geq \frac{m+1}{p+q} |A_{ij}| - C_m,$$ by the induction hypothesis or $A_{ij}^{\prime}$ is FD mod $p$ and FD mod $q$. In the latter case, by \eqref{first} $$|p \cdot A_{ij} + q \cdot A_{ij}| = |p \cdot A_{ij}^{\prime} + q \cdot A_{ij}^{\prime}| \geq (p+q) |A_{ij}| - pq \geq \frac{m+1}{p+q} |A_{ij}| - C_m,$$ since $C_m \geq pq$ and $(p+q) \geq \frac{m+1}{p+q}$. Note that this is the point which blocks us from gradually improving the lower bound beyond $(p+q)|A|$. 

In either case $|p \cdot Q_j + q \cdot P_i| \geq \frac{m+1}{p+q} |A_{ij}| - C_m$. By \eqref{sum}, it follows that

\begin{align*}
|p \cdot A &+ q \cdot A| = \displaystyle \sum_{i = 1}^r \sum_{j=1}^s  |p \cdot Q_j + q \cdot P_i|\\
& \geq \displaystyle \sum_{i = 1}^r \sum_{j=1}^s\left(\frac{m+1}{p + q} |A_{ij}| - C_m\right)\geq \frac{m+1}{p+q} |A| - C_{m+1},
\end{align*} since $C_{m+1} = pq C_m$. 
\end{proof}

\end{document}